\documentclass[a4paper,11pt]{elsarticle}
\usepackage{amsthm}
\usepackage{amsmath}
\usepackage{amssymb}
\usepackage{color}
\usepackage{enumerate}
\usepackage{float}   
\usepackage{subfig}
\usepackage{tikz}
\usepackage{enumerate}
\newcommand\junk[1]{}
\junk{
\makeatletter
\def\ps@pprintTitle{%
  \let\@oddhead\@empty
  \let\@evenhead\@empty
  \def\@oddfoot{\reset@font\hfil\thepage\hfil}
  \let\@evenfoot\@oddfoot
}
\makeatother
}

\newtheorem{theorem}{Theorem}[section]
\newtheorem{prop}[theorem]{Proposition}

\newtheorem{lemma}[theorem]{Lemma}
\newtheorem{observation}[theorem]{Observation}
\newtheorem{conjecture}[theorem]{Conjecture}

\def\I{\mathcal{I}}
\def\J{\mathcal{J}}

\def\bound{24}
\begin{document}
\begin{frontmatter}
\title{Navigating Between Packings of Graphic Sequences\tnoteref{t1}}
\author[renyi]{P\'eter L. Erd\H os\fnref{elp}}
\address[renyi]{Alfr\'ed R{\'e}nyi Institute, Re\'altanoda u 13-15 Budapest, 1053 Hungary\\  {\tt email:} erdos.peter@renyi.mta.hu}
\author[denver]{Michael Ferrara\fnref{mjf}}
\author[denver]{Stephen G. Hartke\fnref{sgh}}
\address[denver]{Department of Mathematical and Statistical Sciences, University of Colorado Denver,US\\
{\tt email:} $<$michael.ferrara; stephen.hartke$>$@ucdenver.edu}
\fntext[elp]{PLE was supported in part by the National Research, Development and Innovation Office -- NKFIH grant K~116769 and SNN~116095.}
\fntext[sgh]{SGH was supported in part by a Collaboration Grant from the Simons Foundation (\#316262 to Stephen G. Hartke).}
\fntext[mjf]{MF was supported in part by Collaboration Grants from the Simons Foundation (\#206692 and \#426971 to Michael Ferrara).}
\tnotetext[t1]{The authors express their gratitude to Tam\'as R. Mezei (R\'enyi Institute, Budapest) for his help in preparing the manuscript.}

\begin{abstract}
Let $\pi_1=(d_1^{(1)}, \ldots,d_n^{(1)})$ and $\pi_2=(d_1^{(2)},\ldots,d_n^{(2)})$ be graphic sequences. We say they \emph{pack} if there exist edge-disjoint realizations $G_1$ and $G_2$ of $\pi_1$ and $\pi_2$, respectively, on vertex set $\{v_1,\dots,v_n\}$ such that for $j\in\{1,2\}$, $d_{G_j}(v_i)=d_i^{(j)}$ for all $i\in\{1,\ldots,n\}$. In this case, we say that $(G_1,G_2)$ is a $(\pi_1,\pi_2)$-\textit{packing}.

A clear necessary condition for graphic sequences $\pi_1$ and $\pi_2$ to pack is that $\pi_1+\pi_2$, their componentwise sum, is also graphic.  It is known, however, that this condition is not sufficient, and furthermore that the general problem of determining if two sequences pack is $NP$-complete.  S.~Kundu proved in 1973 that if $\pi_2$ is almost regular, that is each element is from $\{k-1, k\}$, then $\pi_1$ and $\pi_2$ pack if and only if $\pi_1+\pi_2$ is graphic.

In this paper we will consider graphic sequences $\pi$  with the property that $\pi+\mathbf{1}$ is graphic. By Kundu's theorem, the sequences $\pi$ and $\mathbf{1}$ pack, and there exist edge-disjoint realizations $G$ and $\mathcal{I}$, where $\mathcal{I}$ is a 1-factor. We call such a $(\pi,\mathbf{1})$ packing a {\em Kundu realization}.

Assume that $\pi$ is a graphic sequence, in which each term is at most $n/\bound$, that packs with $\mathbf{1}$. This paper contains two results. On one hand, any two Kundu realizations of the degree sequence $\pi+\mathbf{1}$ can be transformed into each other through a sequence of other Kundu realizations by swap operations. On the other hand, the same conditions ensure that any particular 1-factor can be part of a Kundu realization of $\pi+\mathbf{1}$.
\end{abstract}
\end{frontmatter}
\section{Introduction}
A nonnegative integer sequence $\pi$ is \textit{graphic} if it is the degree sequence of some simple graph $G$ on vertex set $V=\{v_1,\dots,v_n\}$. Unless explicitly stated, we do not assume that the elements of $\pi$ are ordered in any particular way. In this case we say that $G$ \textit{realizes} or is a \textit{realization} of $\pi$.  Graphic sequences $\pi_1=\left (d_1^{(1)}, \ldots,d_n^{(1)}\right )$ and $\pi_2=\left (d_1^{(2)}, \ldots,d_n^{(2)}\right )$ \emph{pack} if there exist edge-disjoint realizations $G_1$ and $G_2$ of $\pi_1$ and $\pi_2$, respectively,  such that for $j\in\{1,2\}$, $d_{G_j}(v_i)=d_i^{(j)}$ for all $i\in\{1,\ldots,n\}$. In this case, we say that $(G_1,G_2)$ is a $(\pi_1,\pi_2)$-\textit{packing}.


The following observation clearly holds when packing graphic sequences:
\begin{observation}\label{Th:trivi}
Assume that graphic sequences $\pi_1$ and $\pi_2$ pack. Then their componentwise sum $\pi=\pi_1+\pi_2$ is also graphic.
\end{observation}
\noindent However it is well known that this condition is not sufficient.

Let $G$ be a fixed but otherwise arbitrary realization of $\pi$. The decision problem of whether $G$ contains a subgraph $G_1$ with degree sequence $\pi_1$ is known to be solvable efficiently---see Tutte, 1954 (\cite{T54}) and Edmonds, 1965 (\cite{E1965a}). On the other hand, the general problem of determining if two sequence pack is \textit{NP}-complete (see D\"{u}rr, Gui\~{n}ez and Matamala, 2012 \cite{NP}).

Although the complexity of the sequence packing problem was unknown until recently, the problem has been studied in specific contexts for some time. Let $\mathbf{k}_n$ denote the sequence in which every term is $k$, which is well-known to be graphic if the sum of the elements of the sequence is even and $k\le n-1$. Going forward, we will suppress the subscript of $n$ when the context is clear.  In 1970 Gr\"unbaum conjectured in \cite{grunbaum} that if $\pi_2$ is the sequence $\mathbf{1}$ consisting of all 1's, then the necessary condition in Observation~\ref{Th:trivi} is also sufficient. This was subsequently generalized by Rao and Rao:
\begin{conjecture}[$k${\bf -factor Conjecture}, A.R. Rao and S.B. Rao, 1972 \cite{rao}]
Graphic degree sequences $\pi_1$ and $\mathbf{k}$ pack if and only if $\pi_1+\mathbf{k}$ is graphic.
\end{conjecture}
\noindent In 1973 S. Kundu proved a more general form of the $k$-factor Conjecture.
\begin{theorem}[Kundu  \cite{kundu}]\label{th:kundu}
Let $\pi_1$ and $\pi_2$ be graphic, where the elements of $\pi_2$  are drawn from $\{k-1, k\}$. Then $\pi_1$ and $\pi_2$ pack if and only if $\pi_1+\pi_2$ is graphic.
\end{theorem}
\noindent In 1974 Lov\'asz gave a shorter proof for the case $\pi_2 = \mathbf{1}$ (see \cite{lovasz}). The ``book proof" for Theorem~\ref{th:kundu} is due to Yang-Chuan Chen (1988, \cite{chen}).

\medskip\noindent
In this paper we consider graphic sequences $\pi$ on the vertex set $V$ with the property that $\pi+\mathbf{1}$ is graphic (so $|V|$ must be even). By Kundu's theorem the sequences $\pi$ and $\mathbf{1}$ pack, and there exist edge-disjoint realizations $G$ and $\mathcal{I}$ where $\mathcal{I}$ is a 1-factor. We call such a realization of the degree sequence $\pi + \mathbf{1}$ a {\em Kundu realization}, and we say that $\mathcal{I}$ is the {\em displayed} 1-factor, as clearly this realization may contain many 1-factors besides $\mathcal{I}$. In particular, there are exponentially many different 1-factors in the complete graph $K_n$. We are interested in which of these 1-factors can occur in Kundu realizations.  Our first result is the following:
\begin{theorem}\label{th:main}
Let $\pi$ be a graphic sequence with $n$ terms in which each term is at most $n/\bound$, and suppose that $\pi$ packs with $\mathbf{1}$. Furthermore, let $\mathcal{J}$ be a given, particular 1-factor on $V$.  Then there exists a realization $G_{\mathcal{J}}$ of $\pi$ such that $(G_{\mathcal{J}},\mathcal{J})$ is a $(\pi,\mathbf{1})$-packing.
\end{theorem}

\medskip \noindent
Let $G$ and $G'$ be two realizations of a graphic sequence $\pi$. A classical result of Petersen (\cite{pet}) states that it is possible to transform $G$ into $G'$ through a sequence of \textit{swap operations}, wherein the edges and non-edges of an alternating 4-cycle are interchanged. Our second result is the following:
\begin{theorem}\label{th:transform}
Let $(G,\mathcal{I})$ and $(G', \mathcal{J})$ be Kundu realizations of  $\pi+\mathbf{1}$.  Then it is possible to transform $(G,\mathcal{I})$ into $(G', \mathcal{J})$ with swap operations via Kundu realizations where in each step the image of the displayed 1-factor is the next displayed 1-factor.
\end{theorem}

\section{Definitions and tools}\label{sec:def}
In this paper all graphs are simple (no multiple edges, no loops) on the vertex set $V$. As already mentioned, the degree sequences of any graphs considered are not assumed to be ordered in any specific way.

Let $G$ be a realization of a graphic sequence $\pi$. If $a,b,c$ and~$d$ are vertices of $G$ satisfying $ab,cd\in E$ and $bc,ad\notin E$, then the graph $G' =(V,E')$ with $E'=E\cup \{bc,ad\} \setminus \{ab,cd\}$ is another realization of~$\pi$. This \textbf{swap} operation, denoted by $ab,cd\Rightarrow bc,ad$, was introduced by Havel~(\cite{havel}) and reinvented by Hakimi~(\cite{hakimi}), and was essentially utilized by Petersen in \cite{pet}. It is also known as a {\it switch} or {\it rewiring} operation.

Let $G$ be a Kundu realization of the packing degree sequences $(\pi_1, \mathbf{1})$ with the displayed 1-factor $\I$. Consider a swap operation $ab,cd \Rightarrow bc,ad$. If
\begin{equation}\label{eq:rest}
\{ab,cd\} \subset \I \quad \mbox{ or } \quad \{ab,cd\} \cap \I=\emptyset
\end{equation}
then the ``image" $\I'=\I \cup \{bc,ad\} \setminus \{ab,cd\}$ is a 1-factor again, and the resulting graph $G'$ is again an $(\pi_1, \mathbf{1})$ packing with the displayed 1-factor $\I'$. However, if $ab \in \I$ but $cd \not \in \I$, then $\I'$ will not contain a 1-factor. It is also possible that $G'$ is not a $(\pi_1, \mathbf{1})$ packing at all.

In this paper we always will apply swap operations which conform to property~(\ref{eq:rest}). Such a swap operation is a {\em Kundu-restricted} swap operation or a {\em K-swap} for short.

\medskip\noindent
Next we recall some important facts about (general) swap sequences from the paper of Erd\H{o}s, Kir\'aly and Mikl\'os (\cite{distance}).  Let $G$ and $G'$ be two realizations of a given (but not necessarily Kundu) degree sequence $\pi$. The symmetric difference $\nabla=E(G) \triangle E(G')$ of their edges has a natural 2-coloring: an edge in $\nabla$ is {\bf red} if it belongs to $G$ and {\bf blue} if it belongs to $G'$.

\begin{theorem}\label{th:decomp}
\begin{enumerate}[{\rm (i)}]
\item Every vertex in $\nabla$ has an equal number of red and blue adjacent edges. Moreover, the symmetric difference $\nabla$ can be decomposed into alternating (with respect to the coloring) closed walks of  even length.
\item When the degree sequence is bipartite, then the decomposition can be made into cycles, where no cycle contains any vertex twice.
\end{enumerate}
\end{theorem}
\begin{theorem}\label{th:swap-series}
\begin{enumerate}[{\rm (i)}]
\item If $\nabla$ consists of only a single alternating cycle $C$, then there exists a sequence of consecutive swap operations transforming a realization $G$ into a realization $G'$ such that every swap in the process is applied to vertex pairs {\rm i.e. ({\bf chords})} contained completely within $V(C)$.
\item The process never uses a chord $uv$ from $C$ if the distance between $u$ and $v$ is even when traversing $C$ from $u$ to $v$ in both directions.  Any chord not satisfying this condition is called {\bf eligible}.
\end{enumerate}
\end{theorem}

\section{Proofs}\label{sec:poof}
Going forward, we consider a graphic sequence $\pi=(d_1,\ldots, d_n)$ such that $\chi:=\pi +\mathbf{1}$ is also graphic and $\Delta:=\max(\pi)\le \frac{n}{\bound}$. Our proofs rely on several lemmas. We start with the following seemingly simple situation:

\begin{lemma}\label{th:disjoint}
Let $(G,\I)$ and $(G,\J)$ be two Kundu realizations of $\chi$ where $\I$ and $\J$ are edge disjoint. Then there exists a K-swap sequence between the two realizations such that $G$ does not change at any step of the process.
\end{lemma}

\medskip\noindent
{\sc Proof:} To begin, color the edges of $G$, $\I$ and $\J$ with {\em green}, {\em red} and {\em blue}, respectively. Consider the symmetric difference $\nabla= \I \triangle \J$. Since $\I$ and $\J$ are edge-disjoint 1-factors, $\nabla$ is a union of pairwise vertex-disjoint alternating red-blue cycles.

Our strategy is to transform $(G,\I)$ with K-swaps into a realization $(G,\I')$, never using green edges, such that $\nabla'=\I' \triangle \J$ has no green eligible chords.  If, by chance, none of these alternating cycles contains a green eligible chord, then the direct application of Theorem~\ref{th:swap-series} (ii) provides the required K-swap sequence.

In every step we select a longest alternating cycle $C$ from $\nabla$ and work to decrease the length of that cycle.  As we proceed, we will use three different processes. The first two processes can be used to decrease the length of $C$. The third will be used when there are many 4-cycles in $\nabla$. While this third process may further decrease the length of $C$, its main purpose is to decrease the number of the eligible green chords.  As would be expected, throughout the application of these processes we may change the red edges in our realization as we transform $\I$ into $\J$.

\medskip\noindent

Let $L$ denote the length of $C$, and let $e$ and $f$ be red edges in $C$ that are not consecutive along $C$. This implies that the length of $C$ is at least 6, as an alternating 4-cycle cannot have any eligible chords.  Assume that some pair of non-crossing chords $\alpha$ and $\beta$ between the corresponding endpoints of $e$ and $f$ are not (green) edges of $G$. (See Figure~\ref{fig:simp1}.) Then the K-swap $e,f \Rightarrow \alpha,\beta$ splits $C$ into two shorter cycles $\Gamma$ and $\Sigma$, increasing the number of cycles in $\nabla$. We call this operation {\bf Process $\clubsuit$}.
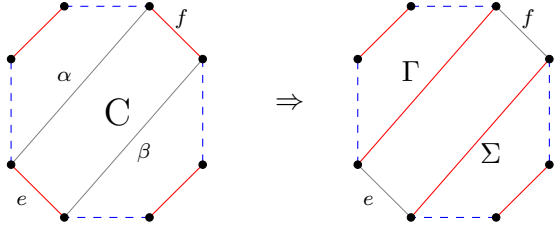
\begin{figure}[h]
\centering
\subfloat{
\begin{tikzpicture}[scale=1.4]
\draw [red] (0,0) -- (-.5,.5);
\node  at (-.4,.4) [label=below:$\scriptstyle e$] {};
\draw [blue, dashed] (-.5,.5) -- (-.5,1.5);
\draw [red](-.5,1.5) -- (0,2);
\draw [blue, dashed] (0,2) -- (.8,2);
\draw [red] (.8,2) -- (1.3,1.5);
\node  at (1.1,1.6) [label=above:$\scriptstyle f$] {};
\draw [blue, dashed] (1.3,1.5) -- (1.3,.5);
\draw [red] (1.3,.5) -- (.8,0);
\draw [blue, dashed] (.8,0) -- (0,0);
\draw [line width=.3pt,gray] (-.5,.5) -- (.8,2);
\node at (.5,1) {\Large C};
\node  at (0,1.1) [label=above:$\scriptstyle \alpha$] {};
\draw [line width=.3pt,gray] (1.3,1.5) -- (0,0);
\node  at (.75,.9) [label=below:$\scriptstyle \beta$] {};
\draw[color=black,fill=black] (0,0) circle (1pt);
\draw [color=black,fill=black] (-.5,.5) circle (1pt);
\draw[color=black,fill=black] (-.5,1.5) circle (1pt);
\draw[color=black,fill=black] (0,2) circle (1pt);
\draw [color=black,fill=black] (.8,2) circle (1pt);
\draw [color=black,fill=black] (1.3,1.5) circle (1pt);
\draw [color=black,fill=black] (1.3,.5) circle (1pt);
\draw [color=black,fill=black] (.8,0) circle (1pt);
\end{tikzpicture}
    }
\qquad \raisebox{1.5cm}{$\displaystyle\Rightarrow$} \quad
\subfloat{
\begin{tikzpicture}[scale=1.4]
\draw [line width=.2pt,gray] (0,0) -- (-.5,.5);
\node at (-.4,.4) [label=below:$\scriptstyle e$] {};
\draw [blue, dashed] (-.5,.5) -- (-.5,1.5);
\draw [red](-.5,1.5) -- (0,2);
\draw [blue, dashed] (0,2) -- (.8,2);
\draw [line width=.2pt,gray] (.8,2) -- (1.3,1.5);
\node  at (1.1,1.6) [label=above:$\scriptstyle f$] {};
\draw [blue, dashed] (1.3,1.5) -- (1.3,.5);
\draw [red] (1.3,.5) -- (.8,0);
\draw [blue, dashed] (.8,0) -- (0,0);
\draw [red] (-.5,.5) -- (.8,2);
\node  at (0,1.1) [label=above:$\Gamma$] {};
\draw [red] (1.3,1.5) -- (0,0);
\node  at (.75,.9) [label=below:$\Sigma$] {};
\draw[color=black,fill=black] (0,0) circle (1pt);
\draw [color=black,fill=black] (-.5,.5) circle (1pt);
\draw[color=black,fill=black] (-.5,1.5) circle (1pt);
\draw[color=black,fill=black] (0,2) circle (1pt);
\draw [color=black,fill=black] (.8,2) circle (1pt);
\draw [color=black,fill=black] (1.3,1.5) circle (1pt);
\draw [color=black,fill=black] (1.3,.5) circle (1pt);
\draw [color=black,fill=black] (.8,0) circle (1pt);
\end{tikzpicture}
    }
\caption{{\bf Process $\clubsuit$}: One cycle K-swapped into two.}
\label{fig:simp1}
\end{figure}

\begin{prop}\label{th:no-long-cycles}
If $ L > 4\Delta + 6$, then we can apply Process $\clubsuit$ to $C$.
\end{prop}
\begin{proof}

Suppose that there is no pair of red edges $e$ and $f$ that permit Process $\clubsuit$. Then every pair of non-consecutive red edges in $C$ must have an eligible green edge between them. Notice that the eligible chords between non-consecutive red pairs are distinct, so there must be at least as many edges of $G$ that are eligible chords of $C$ as there are pairs of non-consecutive red edges in $C$.

Note that the number of red edges in $C$ is $L/2$.
Thus, the number of pairs of non-consecutive red edges in $C$ is \begin{equation}\label{eq:elig-green}
\binom{L/2}{2}-\frac{L}{2} = \frac{L^2}{8} -\frac{3L}{4}.
\end{equation}
However, by the maximum degree condition, there are at most $\frac{L\Delta} {2}$ edges of $G$ that are eligible chords of $C$.  Thus if no Process $\clubsuit$ applies, then
\begin{equation}\label{eq:cyclelength}
\binom{L/2}{2}-L/2 = \frac{L^2}{8} -\frac{3L}{4} \le \frac{L\Delta}{2},
\end{equation}
which simplifies to $L\le 4\Delta+6$, a contradiction.
\end{proof}
As we proceed, we assume that for every longest cycle $C$, Process $\clubsuit$ cannot be applied.

\bigskip\noindent
Now assume that there is some cycle $D$ from $\nabla$ with $|D| \ge 6$.  Let $a$, $b$ and $c$ be red edges in $C$ with endpoints $a_1, a_2, b_1, b_2, c_1, c_2$ in cyclic order, and let $\alpha, \beta$ and $\gamma$ be red edges in $D$ with endpoints $\alpha_1, \alpha_2, \beta_1, \beta_2, \gamma_1, \gamma_2$ in cyclic order. Furthermore, assume that there are no green edges between $C$ and $D$ with endpoints in these triples of edges.  Performing the K-swap $a_1 a_2 ,\alpha_1 \alpha_2 \Rightarrow a_1 \alpha_1 , a_2 \alpha_2$ as well as the similar K-swaps operations for the other two edge pairs provides a new realization $G'$. When performing the three swaps, no green and no blue edges change. In the new symmetric difference $\nabla'$ with $\mathcal{J}$, the cycles $C$ and $D$ are substituted with the three red-blue cycles $\Gamma, \Sigma$ and $\Lambda$ as depicted in Figures~\ref{fig:simp2} and \ref{fig:simp25}. We call this operation {\bf Process $\diamondsuit$.}

\begin{figure}[h]
\subfloat{
\begin{tikzpicture}
\draw [line width=.2,gray] (1.3,.5) -- (2.8,.5);
\draw [line width=.2,gray] (1.3,1.5) -- (2.8,1.5);
\draw [line width=.2,gray] (.8,2) .. controls (2,2.2) .. (3.3,2);
\draw [line width=.2,gray] (0,2)  .. controls (1.6,2.5) and (2.6,2.5) .. (4.1,2);
\draw [line width=.2,gray] (.8,0)  .. controls (2,-.2) ..  (3.3,0);
\draw [line width=.2,gray] (0,0)  .. controls (1.6,-.8) and (2.6,-.8) ..  (4.1,0);

\draw [blue,dashed](0,0) -- (0,2);
\draw [red] (0,2) -- (.8,2);
\draw [blue,dashed] (.8,2) -- (1.3,1.5);
\draw [red] (1.3,1.5) -- (1.3,.5);
\draw [blue,dashed] (1.3,.5) -- (.8,0);
\draw [red] (.8,0) -- (0,0);
\draw[color=black,fill=black] (0,0) node () [label=left:{$\scriptstyle a_1$}] {} circle (1pt);
\draw[color=black,fill=black] (0,2) node () [label=left:{$\scriptstyle c_2$}] {}circle (1pt);
\draw [color=black,fill=black] (.8,2) node () [label=above:{$\scriptstyle c_1$}] {} circle (1pt);
\draw [color=black,fill=black] (1.3,1.5) node () [label=left:{$\scriptstyle b_2$}] {}circle (1pt);
\draw [color=black,fill=black] (1.3,.5) node () [label=left:{$\scriptstyle b_1$}] {}circle (1pt);
\draw [color=black,fill=black] (.8,0) node () [label=below:{$\scriptstyle a_2$}] {} circle (1pt);
\node at (.5,1) {\Large C};
\node at (3.5,1) {\Large D};
\draw [blue,dashed](4.1,0) -- (4.1,2);
\draw [red] (3.3,2) -- (4.1,2);
\draw [blue,dashed] (3.3,2) -- (2.8,1.5);
\draw [red] (2.8,1.5) -- (2.8,.5);
\draw [blue,dashed] (2.8,.5) -- (3.3,0);
\draw [red] (3.3,0) -- (4.1,0);
\draw[color=black,fill=black] (4.1,0) node () [label=right:{$\scriptstyle \alpha_1$}] {} circle (1pt);
\draw[color=black,fill=black] (4.1,2) node () [label=above:{$\scriptstyle \gamma_2$}] {}circle (1pt);
\draw [color=black,fill=black] (3.3,2) node () [label=above:{$\scriptstyle \gamma_1$}] {} circle (1pt);
\draw [color=black,fill=black] (2.8,1.5) node () [label=right:{$\scriptstyle \beta_2$}] {}circle (1pt);
\draw [color=black,fill=black] (2.8,.5) node () [label=right:{$\scriptstyle \beta_1$}] {}circle (1pt);
\draw [color=black,fill=black] (3.3,0) node () [label=below:{$\scriptstyle \alpha_2$}] {} circle (1pt);
\end{tikzpicture}
    }
\quad \raisebox{1.5cm}{$\displaystyle\Rightarrow$} \quad
\subfloat{
\begin{tikzpicture}
\draw [red] (1.3,.5) -- (2.8,.5);
\draw [red] (1.3,1.5) -- (2.8,1.5);
\draw [red] (.8,2) .. controls (2,2.2) .. (3.3,2);
\draw [red] (0,2)  .. controls (1.6,2.5) and (2.6,2.5) .. (4.1,2);
\draw [red] (.8,0)  .. controls (2,-.2) ..  (3.3,0);
\draw [red] (0,0)  .. controls (1.6,-.8) and (2.6,-.8) ..  (4.1,0);

\draw [blue,dashed](0,0) -- (0,2);
\draw [line width=.2,gray] (0,2) -- (.8,2);
\draw [blue,dashed] (.8,2) -- (1.3,1.5);
\draw [line width=.2,gray] (1.3,1.5) -- (1.3,.5);
\draw [blue,dashed] (1.3,.5) -- (.8,0);
\draw [line width=.2,gray] (.8,0) -- (0,0);
\draw[color=black,fill=black] (0,0) node () [label=left:{$\scriptstyle a_1$}] {} circle (1pt);
\draw[color=black,fill=black] (0,2) node () [label=left:{$\scriptstyle c_2$}] {}circle (1pt);
\draw [color=black,fill=black] (.8,2) node () [label=above:{$\scriptstyle c_1$}] {} circle (1pt);
\draw [color=black,fill=black] (1.3,1.5) node () [label=left:{$\scriptstyle b_2$}] {}circle (1pt);
\draw [color=black,fill=black] (1.3,.5) node () [label=left:{$\scriptstyle b_1$}] {}circle (1pt);
\draw [color=black,fill=black] (.8,0) node () [label=below:{$\scriptstyle a_2$}] {} circle (1pt);
\draw [blue,dashed](4.1,0) -- (4.1,2);
\draw [line width=.2,gray] (3.3,2) -- (4.1,2);
\draw [blue,dashed] (3.3,2) -- (2.8,1.5);
\draw [line width=.2,gray] (2.8,1.5) -- (2.8,.5);
\draw [blue,dashed] (2.8,.5) -- (3.3,0);
\draw [line width=.2,gray] (3.3,0) -- (4.1,0);
\draw[color=black,fill=black] (4.1,0) node () [label=right:{$\scriptstyle \alpha_1$}] {} circle (1pt);
\draw[color=black,fill=black] (4.1,2) node () [label=above:{$\scriptstyle \gamma_2$}] {}circle (1pt);
\draw [color=black,fill=black] (3.3,2) node () [label=above:{$\scriptstyle \gamma_1$}] {} circle (1pt);
\draw [color=black,fill=black] (2.8,1.5) node () [label=right:{$\scriptstyle \beta_2$}] {}circle (1pt);
\draw [color=black,fill=black] (2.8,.5) node () [label=right:{$\scriptstyle \beta_1$}] {}circle (1pt);
\draw [color=black,fill=black] (3.3,0) node () [label=below:{$\scriptstyle \alpha_2$}] {} circle (1pt);
\end{tikzpicture}
    }
\caption{{\bf Process $\diamondsuit$}: Two cycles K-swapped into three.}
\label{fig:simp2}
\end{figure}
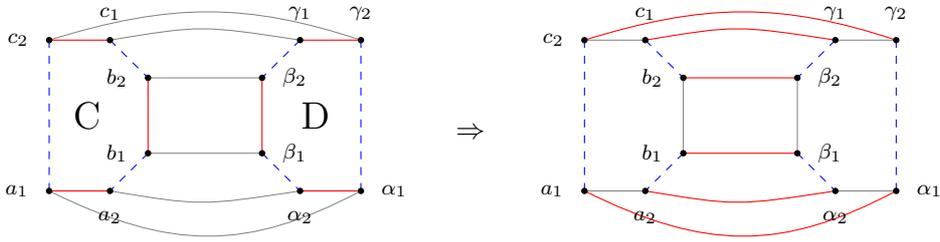

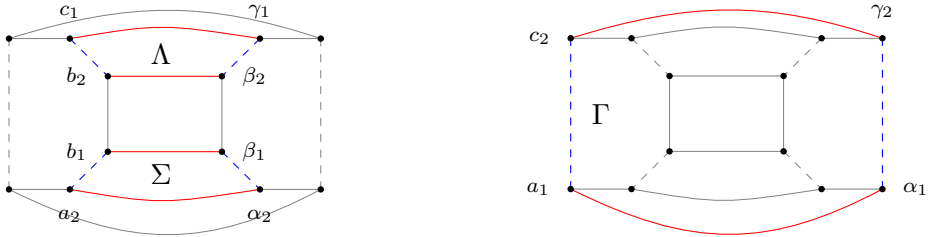
\begin{figure}[h]
\subfloat{
\begin{tikzpicture}
\draw [red] (1.3,.5) -- (2.8,.5);
\draw [red] (1.3,1.5) -- (2.8,1.5);
\draw [red] (.8,2) .. controls (2,2.2) .. (3.3,2);
\draw [line width=.2,gray] (0,2)  .. controls (1.6,2.5) and (2.6,2.5) .. (4.1,2);
\draw [red] (.8,0)  .. controls (2,-.2) ..  (3.3,0);
\draw [line width=.2,gray] (0,0)  .. controls (1.6,-.8) and (2.6,-.8) ..  (4.1,0);
\draw [line width=.2,gray,dashed](0,0) -- (0,2);
\draw [line width=.2,gray] (0,2) -- (.8,2);
\draw [blue,dashed] (.8,2) -- (1.3,1.5);
\draw [line width=.2,gray] (1.3,1.5) -- (1.3,.5);
\draw [blue,dashed] (1.3,.5) -- (.8,0);
\draw [line width=.2,gray] (.8,0) -- (0,0);
\draw[color=black,fill=black] (0,0) node () [label=left:{}] {} circle (1pt);
\draw[color=black,fill=black] (0,2) node () [label=left:{}] {}circle (1pt);
\draw [color=black,fill=black] (.8,2) node () [label=above:{$\scriptstyle c_1$}] {} circle (1pt);
\draw [color=black,fill=black] (1.3,1.5) node () [label=left:{$\scriptstyle b_2$}] {}circle (1pt);
\draw [color=black,fill=black] (1.3,.5) node () [label=left:{$\scriptstyle b_1$}] {}circle (1pt);
\draw [color=black,fill=black] (.8,0) node () [label=below:{$\scriptstyle a_2$}] {} circle (1pt);
\node at (2,.2) {$\displaystyle  \Sigma$};
\node at (2,1.8) {$\displaystyle \Lambda$};
\draw [line width=.2,gray,dashed](4.1,0) -- (4.1,2);
\draw [line width=.2,gray] (3.3,2) -- (4.1,2);
\draw [blue,dashed] (3.3,2) -- (2.8,1.5);
\draw [line width=.2,gray] (2.8,1.5) -- (2.8,.5);
\draw [blue,dashed] (2.8,.5) -- (3.3,0);
\draw [line width=.2,gray] (3.3,0) -- (4.1,0);
\draw[color=black,fill=black] (4.1,0) node () [label=right:{}] {} circle (1pt);
\draw[color=black,fill=black] (4.1,2) node () [label=above:{}] {}circle (1pt);
\draw [color=black,fill=black] (3.3,2) node () [label=above:{$\scriptstyle \gamma_1$}] {} circle (1pt);
\draw [color=black,fill=black] (2.8,1.5) node () [label=right:{$\scriptstyle \beta_2$}] {}circle (1pt);
\draw [color=black,fill=black] (2.8,.5) node () [label=right:{$\scriptstyle \beta_1$}] {}circle (1pt);
\draw [color=black,fill=black] (3.3,0) node () [label=below:{$\scriptstyle \alpha_2$}] {} circle (1pt);
\end{tikzpicture}
    }
\qquad \qquad \quad
\subfloat{
\begin{tikzpicture}
\draw [line width=.2,gray] (1.3,.5) -- (2.8,.5);
\draw [line width=.2,gray] (1.3,1.5) -- (2.8,1.5);
\draw [line width=.2,gray] (.8,2) .. controls (2,2.2) .. (3.3,2);
\draw [red] (0,2)  .. controls (1.6,2.5) and (2.6,2.5) .. (4.1,2);
\draw [line width=.2,gray] (.8,0)  .. controls (2,-.2) ..  (3.3,0);
\draw [red] (0,0)  .. controls (1.6,-.8) and (2.6,-.8) ..  (4.1,0);
\draw [blue,dashed](0,0) -- (0,2);
\draw [line width=.2,gray] (0,2) -- (.8,2);
\draw [line width=.2,gray,dashed] (.8,2) -- (1.3,1.5);
\draw [line width=.2,gray] (1.3,1.5) -- (1.3,.5);
\draw [line width=.2,gray,dashed] (1.3,.5) -- (.8,0);
\draw [line width=.2,gray] (.8,0) -- (0,0);
\draw[color=black,fill=black] (0,0) node () [label=left:{$\scriptstyle a_1$}] {} circle (1pt);
\draw[color=black,fill=black] (0,2) node () [label=left:{$\scriptstyle c_2$}] {}circle (1pt);
\draw [color=black,fill=black] (.8,2) node () [label=above:{}] {} circle (1pt);
\draw [color=black,fill=black] (1.3,1.5) node () [label=left:{}] {}circle (1pt);
\draw [color=black,fill=black] (1.3,.5) node () [label=left:{}] {}circle (1pt);
\draw [color=black,fill=black] (.8,0) node () [label=below:{}] {} circle (1pt);
\node at (.4,1) {$\displaystyle \Gamma$};
\draw [blue,dashed](4.1,0) -- (4.1,2);
\draw [line width=.2,gray] (3.3,2) -- (4.1,2);
\draw [line width=.2,gray,dashed] (3.3,2) -- (2.8,1.5);
\draw [line width=.2,gray] (2.8,1.5) -- (2.8,.5);
\draw [line width=.2,gray,dashed] (2.8,.5) -- (3.3,0);
\draw [line width=.2,gray] (3.3,0) -- (4.1,0);
\draw[color=black,fill=black] (4.1,0) node () [label=right:{$\scriptstyle \alpha_1$}] {} circle (1pt);
\draw[color=black,fill=black] (4.1,2) node () [label=above:{$\scriptstyle \gamma_2$}] {}circle (1pt);
\draw [color=black,fill=black] (3.3,2) node () [label=above:{}] {} circle (1pt);
\draw [color=black,fill=black] (2.8,1.5) node () [label=right:{}] {}circle (1pt);
\draw [color=black,fill=black] (2.8,.5) node () [label=right:{}] {}circle (1pt);
\draw [color=black,fill=black] (3.3,0) node () [label=below:{}] {} circle (1pt);
\end{tikzpicture}
    }
\caption{{\bf Process $\diamondsuit$}: The three new cycles}
\label{fig:simp25}
\end{figure}

Assume now that we cannot execute Process~$\diamondsuit$ for cycles $C$ and $D$. Then every set of three red edges in $C$ must be connected with a green edge to every set of three red edges in $D$. How many green edges are necessary to achieve that? Instead of answering this question in full generality, we will use a relatively easy lower bound of the number of required green edges.

Assume now that there is exactly one green edge between $C$ and $D$ with endpoints in these triples of edges. Then maybe the Process~$\diamondsuit$ cannot execute on those edges as we described above. Without loss of generality, we may assume that this green edge is $b_1\beta_1$. But then consider the other possible 1-factor between end points $b_1, b_2$ and $\beta_1, \beta_2$. If we consider the opposite orientation of the edges of cycle $D$, then Process~$\diamondsuit$ goes through without the slightest problem. So we need at least two green edges between those three red edges in cycle $C$ and the other three edges in cycle $D$ to deny this process.

\medskip\noindent So assume that our cycle $C$ has maximum length $L=2a$ and $D=2b \ge 6$.
\begin{prop}\label{th:noD}
Assume that Process~$\diamondsuit$ cannot be applied for cycles $C$ and $D$. Then there are at least $N=(a-1) b/3$ green edges between the cycles.
\end{prop}
\begin{proof}
Partition the red edges of $D$ into as many triplets as possible. Depending on the remainder of $b$ modulo 3, we distinguish three cases.
\begin{enumerate}[{\rm (1)}]
\setcounter{enumi}{-1}
\item In this case $N$ is well defined, and the statement follows from the reasoning above.
\item Fix a triplet $T$ and consider the edge $e$ in $D$ that does not belong to any triplet.  We know that there are at least $a-1$ green edges between $T$ and $C$. Fix a set $S$ of $a-1$ edges from these green edges. By the pigeonhole principle, there are two edges from $T$, say $\varepsilon$ and $\varphi$, that together they are connected to $C$ with no more than $2(a-1)/3$ green edges from $S$. Therefore there are at last $a - 2(a-1)/3$ red edges in $C$ which are not connected at all to any of $e, \varepsilon, \varphi$. This explains why we need $(a-1)/3$ extra green edges from this triplet to $C$.
\item In this case, let $e$ and $f$ be the edges that are not contained in any triplet. Now in our fixed triplet $T$ there is an edge $\varepsilon$ with not more than $(a-1)/3$ green edges to $C$ (from $S$). Considering the other $a-(a-1)/)3$ red edges from $C$ and the triplet $e,f,\varepsilon$, we need $a-(a-1)/3 -2$ more green edges connecting them to $C$.
\end{enumerate}
\end{proof}
\noindent So if Process~$\diamondsuit$ does not apply for $C$ and $D$, we have at least $\frac{(L/2-1)(|D|/2)}{3}$ green edges between them. Denote
\begin{equation}\label{eq:sum}
Z:= \sum \left \{ |D| : D \mbox{ is a cycle in } (\nabla \setminus C),\  |D|\ge 6 \right \}.
\end{equation}
If Process $\diamondsuit$ cannot be applied for $C$ and for any other {\bf long} cycle (cycles of length at least 6) then at least
\begin{equation}\label{eq:noD}
\frac{(L/2-1)(Z/2)}{3}
\end{equation}
green edges are within $C$ and between $C$ and the other long cycles.

\medskip\noindent
Assume that neither Process $\clubsuit$ nor $\diamondsuit$ apply for cycle $C$. Then, by inequality (\ref{eq:elig-green}), denying $\clubsuit$ requires $L^2 /8 - 3L/4$ well-placed eligible green edges in $C$. Each of these edges is adjacent with two vertices from $C$. Therefore it uses $L^2/4 - 3L/2$ degrees from $C$.

Putting together this with the bound in (\ref{eq:noD}): if we cannot increase the number of cycles in $\nabla$ by applying Processes $\clubsuit$ or $\diamondsuit$ for $C$ then we have:

\begin{align}\label{eq:summing}
&\frac{(L/2-1)(Z/2)}{3} + \frac{L^2}{4} - \frac{3L}{2} \le L\Delta, \quad\mbox{so} \nonumber \\[2pt]
&(L-2) Z \le L(12\Delta - 3L + 18).
\end{align}
For convenience we will use a slightly weaker upper bound pair on $Z$.
\begin{equation}\label{eq:noDa}
Z \le \begin{cases}
16 \Delta -4L +24 < 16 \Delta & \text{if } L\ge 8 \\
18\Delta & \text{if } L=6.
\end{cases}
\end{equation}

\bigskip\noindent
As we proceed, we will use the 4-cycles in $\nabla$ to define a third process that has two possible purposes. This process may decrease the number of eligible green chords for the case $L=6$. Further, if the maximum cycle length $L\ge 8$ then it will decrease the number of cycles in $\nabla$ of (maximal) size $L$ by 1.

Let $e$ and $f$ be red edges on $C$, and let $\varphi$ and $\varepsilon$ be the red edges on the 4-cycle $D$.  Either K-swap on the pair $f,\varphi$ with the corresponding non-edges (one such pair is shown by grey edges in Figure~\ref{fig:simp3}\subref{s2}) will combine $C$ and $D$ into one cycle, which we will call $C'$.  Let $e=aa'$ and $\varepsilon=bb'$ and assume without loss of generality that $a,a',b',b$ appear in that order when traversing $C$ from $a$ to $b$.  Performing the $K$-swap on $e$ and $\varepsilon$ with the non-edges $a'b'$ and $ab$ then results in two alternating cycles.  We call the execution of these two K-swaps {\bf Process $\spadesuit$}.

We then have the following.

\begin{prop}\label{th:4cycle}
Let $D$ be a 4-cycle in $\nabla$.
\begin{enumerate}[{\rm (i)}]
\item Assume that $L=6$ and let $g$ be a green eligible chord in $C$, furthermore assume there is no green edge between $C$ and $D$. Process $\spadesuit$ then transforms $\I$ into $\I'$ such that the new $\nabla'$ has one less green eligible chord.
\item Assume that $L\ge 8$, and there are two nonconsecutive red edges $e, f$ on $C$ with no green edges from those to $D$. Then applying process $\spadesuit$ decreases the number of cycles of length $L$ in $\nabla$. (Then the number of eligible green edges even may increase.)
\end{enumerate}
\end{prop}

\begin{proof}
To establish (i), we choose $e$ and $f$ to be red edges on opposite sides of an eligible green chord $g$, as depicted in Figure \ref{fig:simp3}\subref{s2}.  Part (ii) is an immediate consequence of the Process as defined, although we note that in this case we may not reduce the number of eligible chords.
\end{proof}

\begin{figure}[H]
\centering
\subfloat [][There are 2 disjoint R/B cycles] {\label{s1}
\begin{tikzpicture}
\draw [red] (0,0) -- (-.5,.5);
\node  at (-.2,.4) {$\scriptstyle e$};
\draw [blue, dashed] (-.5,.5) -- (-.5,1.5);
\draw [red](-.5,1.5) -- (0,2);
\draw [blue, dashed] (0,2) -- (.8,2);
\draw [red] (.8,2) -- (1.3,1.5);
\node  at (1,1.6) {$\scriptstyle f$};
\draw [blue, dashed] (1.3,1.5) -- (1.3,.5);
\draw [red] (1.3,.5) -- (.8,0);
\draw [blue, dashed] (.8,0) -- (0,0);
\draw [green!50!black] (0,0) -- (0,2);
\node  at (.15,1) {$\scriptstyle g$};
\draw[color=black,fill=black] (0,0) circle (1pt);
\draw [color=black,fill=black] (-.5,.5) circle (1pt);
\draw[color=black,fill=black] (-.5,1.5) circle (1pt);
\draw[color=black,fill=black] (0,2) circle (1pt);
\draw [color=black,fill=black] (.8,2) circle (1pt);
\draw [color=black,fill=black] (1.3,1.5) circle (1pt);
\draw [color=black,fill=black] (1.3,.5) circle (1pt);
\draw [color=black,fill=black] (.8,0) circle (1pt);
\draw [line width=.2,gray] (.8,2) .. controls (1.9,2) .. (3.5,1.5);
\draw [line width=.2,gray] (1.3,1.5) -- (2.5,1.5);
\draw [line width=.2,gray] (0,0) .. controls (1,-.2) and (1.9,-.2) ..  (2.5,.5);
\draw [line width=.2,gray] (-.5,.5)  .. controls (.5,.5) and (1.9,-.5) ..   (3.5,.5);
\draw [red] (2.5,1.5) -- (3.5,1.5);
\node  at (3,1.3) {$\scriptstyle \varphi$};
\draw [red] (2.5,.5) -- (3.5,.5);
\node  at (3,.7) {$\scriptstyle \varepsilon$};
\draw [blue, dashed] (2.5,.5) -- (2.5,1.5);
\draw [blue, dashed] (3.5,.5) -- (3.5,1.5);

\draw [color=black,fill=black] (2.5,.5) circle (1pt);
\draw [color=black,fill=black] (2.5,1.5) circle (1pt);
\draw [color=black,fill=black] (3.5,1.5) circle (1pt);
\draw [color=black,fill=black] (3.5,.5) circle (1pt);
\end{tikzpicture}
    }
    \qquad\qquad
\subfloat [][After the first K-swap all R/B edges belong to one cycle] {\label{s2}
\begin{tikzpicture}
\draw [red] (0,0) -- (-.5,.5);
\node  at (-.2,.4) {$\scriptstyle e$};
\draw [blue, dashed] (-.5,.5) -- (-.5,1.5);
\draw [red](-.5,1.5) -- (0,2);
\draw [blue, dashed] (0,2) -- (.8,2);
\draw [line width=.2,gray] (.8,2) -- (1.3,1.5);
\node  at (1,1.6) {$\scriptstyle f$};
\draw [blue, dashed] (1.3,1.5) -- (1.3,.5);
\draw [red] (1.3,.5) -- (.8,0);
\draw [blue, dashed] (.8,0) -- (0,0);
\draw [green!50!black] (0,0) -- (0,2);
\node  at (.15,1) {$\scriptstyle g$};
\draw[color=black,fill=black] (0,0) circle (1pt);
\draw [color=black,fill=black] (-.5,.5) circle (1pt);
\draw[color=black,fill=black] (-.5,1.5) circle (1pt);
\draw[color=black,fill=black] (0,2) circle (1pt);
\draw [color=black,fill=black] (.8,2) circle (1pt);
\draw [color=black,fill=black] (1.3,1.5) circle (1pt);
\draw [color=black,fill=black] (1.3,.5) circle (1pt);
\draw [color=black,fill=black] (.8,0) circle (1pt);
\draw [red] (.8,2) .. controls (1.9,2) .. (3.5,1.5);
\draw [red] (1.3,1.5) -- (2.5,1.5);
\draw [line width=.2,gray] (0,0) .. controls (1,-.2) and (1.9,-.2) ..  (2.5,.5);
\draw [line width=.2,gray] (-.5,.5)  .. controls (.5,.5) and (1.9,-.5) ..   (3.5,.5);
\draw [line width=.2,gray] (2.5,1.5) -- (3.5,1.5);
\node  at (3,1.3) {$\scriptstyle \varphi$};
\draw [red] (2.5,.5) -- (3.5,.5);
\node  at (3,.7) {$\scriptstyle \varepsilon$};
\draw [blue, dashed] (2.5,.5) -- (2.5,1.5);
\draw [blue, dashed] (3.5,.5) -- (3.5,1.5);

\draw [color=black,fill=black] (2.5,.5) circle (1pt);
\draw [color=black,fill=black] (2.5,1.5) circle (1pt);
\draw [color=black,fill=black] (3.5,1.5) circle (1pt);
\draw [color=black,fill=black] (3.5,.5) circle (1pt);
\end{tikzpicture}
    }
\caption{{\bf Process $\spadesuit$}:\ Decreasing the number of green eligible chords.}
\label{fig:simp3}
\end{figure}
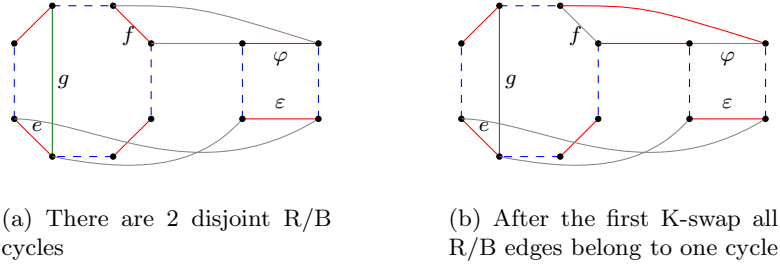
\begin{figure}[H]
\centering
\subfloat{
\begin{tikzpicture}
\draw [line width=.2,gray] (0,0) -- (-.5,.5);
\node  at (-.2,.4) {$\scriptstyle e$};
\draw [line width=.2,gray, dashed] (-.5,.5) -- (-.5,1.5);
\draw [line width=.2,gray](-.5,1.5) -- (0,2);
\draw [line width=.2,gray, dashed] (0,2) -- (.8,2);
\draw [line width=.2,gray] (.8,2) -- (1.3,1.5);
\node  at (1,1.6) {$\scriptstyle f$};
\draw [blue, dashed] (1.3,1.5) -- (1.3,.5);
\draw [red] (1.3,.5) -- (.8,0);
\draw [blue, dashed] (.8,0) -- (0,0);
\draw [green!50!black] (0,0) -- (0,2);
\node  at (.15,1) {$\scriptstyle g$};
\draw[color=black,fill=black] (0,0) circle (1pt);
\draw [color=black,fill=black] (-.5,.5) circle (1pt);
\draw[color=black,fill=black] (-.5,1.5) circle (1pt);
\draw[color=black,fill=black] (0,2) circle (1pt);
\draw [color=black,fill=black] (.8,2) circle (1pt);
\draw [color=black,fill=black] (1.3,1.5) circle (1pt);
\draw [color=black,fill=black] (1.3,.5) circle (1pt);
\draw [color=black,fill=black] (.8,0) circle (1pt);
\draw [line width=.2,gray] (.8,2) .. controls (1.9,2) .. (3.5,1.5);
\draw [red] (1.3,1.5) -- (2.5,1.5);
\draw [red] (0,0) .. controls (1,-.2) and (1.9,-.2) ..  (2.5,.5);
\draw [line width=.2,gray] (-.5,.5)  .. controls (.5,.5) and (1.9,-.5) ..   (3.5,.5);
\draw [line width=.2,gray] (2.5,1.5) -- (3.5,1.5);
\node  at (3,1.3) {$\scriptstyle \varphi$};
\draw [line width=.2,gray] (2.5,.5) -- (3.5,.5);
\node  at (3,.7) {$\scriptstyle \varepsilon$};
\draw [blue, dashed] (2.5,.5) -- (2.5,1.5);
\draw [line width=.2,gray, dashed] (3.5,.5) -- (3.5,1.5);

\draw [color=black,fill=black] (2.5,.5) circle (1pt);
\draw [color=black,fill=black] (2.5,1.5) circle (1pt);
\draw [color=black,fill=black] (3.5,1.5) circle (1pt);
\draw [color=black,fill=black] (3.5,.5) circle (1pt);
\end{tikzpicture}
    }
    \qquad\qquad
\subfloat{
\begin{tikzpicture}
\draw [line width=.2,gray] (0,0) -- (-.5,.5);
\node  at (-.2,.4) {$\scriptstyle e$};
\draw [blue, dashed] (-.5,.5) -- (-.5,1.5);
\draw [red](-.5,1.5) -- (0,2);
\draw [blue, dashed] (0,2) -- (.8,2);
\draw [line width=.2,gray] (.8,2) -- (1.3,1.5);
\node  at (1,1.6) {$\scriptstyle f$};
\draw [line width=.2,gray, dashed] (1.3,1.5) -- (1.3,.5);
\draw [line width=.2,gray] (1.3,.5) -- (.8,0);
\draw [line width=.2,gray, dashed] (.8,0) -- (0,0);
\draw [green!50!black] (0,0) -- (0,2);
\node  at (.15,1) {$\scriptstyle g$};
\draw[color=black,fill=black] (0,0) circle (1pt);
\draw [color=black,fill=black] (-.5,.5) circle (1pt);
\draw[color=black,fill=black] (-.5,1.5) circle (1pt);
\draw[color=black,fill=black] (0,2) circle (1pt);
\draw [color=black,fill=black] (.8,2) circle (1pt);
\draw [color=black,fill=black] (1.3,1.5) circle (1pt);
\draw [color=black,fill=black] (1.3,.5) circle (1pt);
\draw [color=black,fill=black] (.8,0) circle (1pt);
\draw [red] (.8,2) .. controls (1.9,2) .. (3.5,1.5);
\draw [line width=.2,gray] (1.3,1.5) -- (2.5,1.5);
\draw [line width=.2,gray] (0,0) .. controls (1,-.2) and (1.9,-.2) ..  (2.5,.5);
\draw [red] (-.5,.5)  .. controls (.5,.5) and (1.9,-.5) ..   (3.5,.5);
\draw [line width=.2,gray] (2.5,1.5) -- (3.5,1.5);
\node  at (3,1.3) {$\scriptstyle \varphi$};
\draw [line width=.2,gray] (2.5,.5) -- (3.5,.5);
\node  at (3,.7) {$\scriptstyle \varepsilon$};
\draw [line width=.2,gray, dashed] (2.5,.5) -- (2.5,1.5);
\draw [blue, dashed] (3.5,.5) -- (3.5,1.5);

\draw [color=black,fill=black] (2.5,.5) circle (1pt);
\draw [color=black,fill=black] (2.5,1.5) circle (1pt);
\draw [color=black,fill=black] (3.5,1.5) circle (1pt);
\draw [color=black,fill=black] (3.5,.5) circle (1pt);
\end{tikzpicture}
    }
\caption{{\bf Process $\spadesuit$}: The new R/B cycles in $\nabla'$. Edge $g$ is not eligible anymore.}
\label{fig:simp4}
\end{figure}
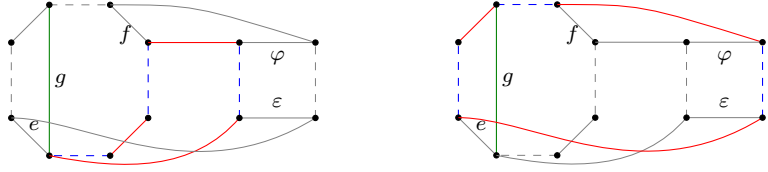

Suppose that $L\ge 8$, and that part (ii) of Proposition \ref{th:4cycle} does not apply to $C$ and some 4-cycle $D$ in $\nabla$, and also assume that some red edge $e$ in $C$ is not connected to $D$ by any green edge. Then all red edges in $C$, except at most one of the two neighboring red edges of $e$,  must be connected to $D$. Both red edges cannot be exempt, because then these two edges would be nonconsecutive on $C$. Therefore, at least $L/2 -2$ red edges on $C$ must have a green edge to $D$. The number of the 4-cycles in $\nabla$ is $F = (n - Z-L)/4$. If none of them is suitable Process $\spadesuit$, then at least
\begin{equation}
\frac{n - Z-L}{4} (L/2-2)\nonumber
\end{equation}
green edges must go from $C$ to 4-cycles. If we further maintain that Process $\diamondsuit$ does not hold for $C$, from (\ref{eq:noD}) we obtain an additional $\frac{(L/2-1)(Z/2)}{3}$ green edges incident to $C$.  We therefor have that
\begin{align}\label{eq:noP}
\frac{n - Z-L}{4} (L/2-2) + \frac{(L/2-1)(Z/2)}{3} &\le  L \Delta \nonumber \\[2pt]
3(n-Z-L)(L/2 -2) + 2 (L/2-1)Z &\le 12 \Delta  \nonumber \\[2pt]
Z(-L/2 + 4) + 3(n-L)(L/2-2) \le 12 L\Delta.
\end{align}
When $L=8$ then we have
\begin{displaymath}
n \le 12\Delta +8,
\end{displaymath}
which, due to our assumption on $\Delta$, does not hold.

When $L> 8$ the coefficient of $Z$ is negative, so substitute $Z$ with an upper bound decreases the LHS of (\ref{eq:noP}). From that the next inequality follows:
\begin{align}
3(n-L)(L/2-2) \le (20L-64)\Delta &< 20(L-3)\Delta \nonumber \\[2pt]
\frac{3}{2} (n-L) < 20 \frac{L-3}{L-4} \Delta &\le \frac{70}{3} \Delta. \nonumber
\end{align}
In the last line we used that $(L-3)/(L-4)$ is monotone decreasing as $L$ is increasing. Now assume that Process $\clubsuit$ does not apply for $C$, therefore $L \le 4\Delta+6$ by Proposition \ref{th:no-long-cycles}. Since on the LHS the coefficient of $L$ is negative, therefore substituting $L$ its its upper bound decreasing the LHS:
\begin{align}\label{eq:end}
\frac{3}{2}(n -4 \Delta -6) \le \frac{70}{6} \nonumber\\[2pt]
n-6 \le \frac {176}{9} \Delta.
\end{align}
Since by our assumption $\Delta \le n/\bound$ this does not hold.
Therefore as long as $L\ge 8$ we always can apply at least one of Processes $\clubsuit, \diamondsuit$  and $\spadesuit$.

\medskip\noindent Now we consider the $L=6$ case in the light of Process $\spadesuit.$ In this case before the Process we have a length 6 and a length 4 cycle, and this will not be changed by the Process. Here we should use the Process to decrease the number of the eligible green chords.

If in $C$ there is a not-green eligible chord, then the canonical swap sequence, started with this non-green chord works. So we may assume that all three eligible chords in $C$ are green. As we saw in Proposition \ref{th:noD}  between $C$ and any other long (that is length 6) cycle $D$ there are at least 2 green edges to deny Process $\spadesuit$. We have $Z/6$ cycles to play the role of $D$ therefore it requires $Z/3$ green edges between $C$ and the long cycles, furthermore 3 green eligible chords in $C$. (However we discard this 6 edges from our calculation.) We have
\begin{align}\label{eq:hat}
\frac{n-Z-6}{4} + \frac{Z}{3} &\le 6\Delta \nonumber\\[2pt]
n + \frac{Z}{3} &\le 24\Delta,
\end{align}
which contradict to our assumption $\Delta \le n/\bound$. So Process $\spadesuit$ (i) eliminates all eligible green edges from the cycles of length 6, which finishes the proof of Lemma~\ref{th:disjoint}. \qed

\begin{lemma}\label{th:non-disjoint}
Let $(G,\I)$ and $(G,\J)$ be two Kundu realizations of $\chi$. Then there exists a K-swap sequence between the two realizations such that $G$ does not change at any step of the process.
\end{lemma}
\begin{proof}
This statement is a direct consequence of Lemma~\ref{th:disjoint}. If $\I$ and $\J$ are not overlapping, then Lemma~\ref{th:disjoint} applies. If this is not the case, then consider the complement of $K = \overline{G \cup \I \cup \J}$. In $\overline{K}$, all vertex degrees are more than $\frac{\bound -1}{\bound}n-2$, which is at least $n/2$ when $n\ge 4$.  Therefore by Dirac's theorem (\cite{dirac}) there exists a Hamiltonian cycle $H$ in $K$. Taking every other edge in $H$ forms a 1-factor $\I'$, which is disjoint from $G$, $\I$ and $\J$. Now Lemma~\ref{th:disjoint} clearly applies for $G$, $\I$ and $\I'$, providing a K-swap sequence that changing $\I$ into $\I'$ without changing $G$. A second application of Lemma~\ref{th:disjoint} transforms $(G,\I')$ into $(G, \J)$.
\end{proof}

\begin{lemma}\label{th:moving}
Let $(G_1,\I)$ and $(G_2,\J)$ two Kundu realizations. Then there is a K-swap sequence transforming the first one into the second one.
\end{lemma}

\begin{proof}
First we will transform $(G_1,\I)$ into a Kundu realization $(G_2,\I')$. Then application of Lemma~\ref{th:non-disjoint} to $(G_2,\I')$ to obtain $(G_2,\J)$ finishes the proof.

Consider a swap sequence which transforms $G_1$ into $G_2$. This sequence is not necessarily a K-swap sequence. Consider the swap operation $S  = e,f \Rightarrow \varepsilon, \varphi$ transforming $G'$ into $G''$. The edges $e$ and $f$ belong to the graph $G'$. If both edges $\varepsilon, \varphi$ are outside of $\I'$, or both belong to $\I'$, then the swap is automatically a K-swap.
We face a problem, if, say, $\varepsilon$ is in $\I'$ but $\varphi$ is not.

Before we execute the swap $S$ we will ``swap out" edge $\varepsilon$ from the 1-factor. To do that we need another edge $\sigma$ from $\I'$ such that there is no green edge (edge from $G'$) between $\varepsilon$ and the newly chosen red edge $\sigma$. This can be done easily since the joint neighborhood of the two end vertices of edge $\varepsilon$ has no more than $2 \Delta \le n/12 $ vertices. That many vertices cannot cover more than $n/6$  red edges from $\I'$. Picking up edge $\sigma$ which is not covered by these vertices provides the swapping out operation. After that we can proceed with $S$, which is now a K-swap.
\end{proof}
\noindent This lemma is just a rewording of Theorem~\ref{th:transform} what is now proved. \qed \hspace{-12pt} \qed

It is worth noting that the ``swap out" option also can be used in the proof of Lemma~\ref{th:non-disjoint} instead of Dirac's theorem.

\medskip\noindent
We are ready now to prove Theorem~\ref{th:main}: Let $(G,\I)$ be a Kundu realization of $\pi+\mathbf{1}$. Furthermore let $\J$ be a given 1-factor. If $\J$ is disjoint from $G$, then we are done.

So assume that edge $e \in \J$ is a green edge of $G$.  Then we will swap out that green edge from $G$. This is actually the same procedure that was used in the proof of Lemma~\ref{th:moving}. We have to find another green edge $f$ in $G\setminus \J$ such that there is no red edges from $\I$ or $\J$ between these $e$ and $f$. Since no (other) edge from $\J$ touches edge $e,$ therefore $f$ should not belong to $\J$ and must avoid red edges from $\I$ which also incident with $e$. But for that end the same enumeration applies. This finishes the proof of Theorem~\ref{th:main}. \qed

\medskip\noindent It did not escape our attention that if degree sequences $\pi_1$ and $\pi_2$ pack - where the maximum degree in $\pi_2$ is small (say $\pi_2 \le \mathbf{4})$ - then the majority of the operations and reasoning can be applied quite easily to the degree sequence $\pi_1 + \pi_2$. However there may be some problem to handle the alternating cycle decomposition of the symmetric difference of two edge-disjoint realizations of $\pi_2$. However we believe that for packing a degree sequence with a relatively small second sequence similar results apply.

\section*{\refname}  

\end{document}